\documentclass[smallextended,numbook,runningheads]{svjour3}     
\smartqed  
\usepackage{graphicx}
\usepackage{amsmath}
\usepackage{mathptmx}      
\usepackage{amssymb}
\usepackage{epstopdf}

\begin{document}


\title{Optimal preconditioning for image deblurring with
Anti-Reflective boundary conditions}

\author{Pietro Dell'Acqua \and  Stefano Serra-Capizzano  \and Cristina Tablino-Possio}


\institute{Pietro Dell'Acqua \at
           Dipartimento di Scienza ed Alta Tecnologia, Universit\`{a} degli Studi dell'Insubria, \\ Via
           Valleggio 11, 22100 Como, Italy.
           \email{pietro.dellacqua@gmail.com}           
\and
            Stefano Serra-Capizzano \at
            Dipartimento di Scienza ed Alta Tecnologia, Universit\`{a} degli Studi dell'Insubria, \\ Via Valleggio 11, 22100 Como, Italy.
           \email{stefano.serrac@uninsubria.it}           
\and
            Cristina Tablino-Possio \at
            Dipartimento di Matematica e Applicazioni, Universit\`a di Milano Bicocca, \\ Via Cozzi 53, 20125 Milano,
            Italy.
            \email{cristina.tablinopossio@unimib.it}
}

\date{
}

\maketitle

\begin{abstract}
Inspired by the theoretical results on optimal preconditioning
stated by Ng, R.Chan, and Tang in the framework of Reflective
boundary conditions (BCs), in this paper we present analogous
results for Anti-Reflective BCs, where an additional technical
difficulty is represented by the non orthogonal character of the
Anti-Reflective transform and indeed the technique of Ng, R.Chan,
and Tang can not be used. Nevertheless, in both cases, the optimal
preconditioner is the blurring matrix associated to the
symmetrized Point Spread Function (PSF). The geometrical idea on
which our proof is based is very simple and general, so it may be
useful in the future to prove theoretical results for new proposed
boundary conditions. Computational results show that the
preconditioning strategy is effective and it is able to give rise
to a meaningful acceleration both for slightly and highly
non-symmetric PSFs.
\end{abstract}


\section{Introduction}
\label{sez:Introduction} Image deblurring problems
\cite{BerBoc,EHN,Hansen,Deblurring} represents an important and
deeply studied example belonging to the wide area of inverse
problems. In its simplest form, the deblurring problem consists in
finding the true image of an unknown object, having only the
detected image, which is affected by blur and noise. In this paper
we deal with the classical image restoration problem of blurred
and noisy images in the case of a space invariant blurring: under
such assumption, the image formation process is modelled according
to the following integral equation with space invariant kernel
\begin{equation}  \label{eq:modello_continuo}
g(x)=\int h(x-\tilde x) f(\tilde x) d\tilde x + \eta(x),\ x \in \mathbb{R}^2,
\end{equation}
where $f$ denotes the true physical object to be restored, $g$ is
the recorded blurred and noisy image, $\eta$ takes into account
unknown errors in the collected data, e.g. measurement errors and
noise. We consider a standard $2D$ generalization of the rectangle
quadrature formula on an equispaced grid, ordered row-wise from
the top-left corner to the bottom-right one, for the
discretization of (\ref{eq:modello_continuo}). Thus, we obtain the
relations
\begin{equation}  \label{eq:modello_discreto_base}
g_{i}=\sum_{j \in \mathbb{Z}^2}h_{i-j}f_{j}+\eta_i,\quad i \in \mathbb{Z}^2,
\end{equation}
in which an infinite and a shift-invariant matrix $\widetilde
A_\infty=[h_{i-j}]_{(i,j)=((i_1,i_2),(j_1,j_2))}$, i.e., a two-level
Toeplitz matrix, is involved.

Though (\ref{eq:modello_discreto_base}) relies in an infinite
summation since the true image scene does not have a finite
boundary, the data $g_i$ are clearly collected only at a finite
number of values, so representing only a finite region of such an
infinite scene. The blurring operator typically shows also a
finite support, so that it is completely described by a Point
Spread Function (PSF) mask such as
\begin{equation}  \label{eq:coefficienti_psf}
h_{PSF} =\left[h_{i_1,i_2}\right]_{i_1=-q_1,\ldots,q_1,i_2=-q_2,\ldots,q_2}
\end{equation}
where $h_{i_1,i_2}\ge 0$ for any $i_1,i_2$ and $\sum_{i=-q}^{q}h_{i} =1$, $%
i=(i_1,i_2)$, $q=(q_1,q_2)$ and the normalization is according to
a suitable conservation law. Therefore, relations
(\ref{eq:modello_discreto_base}) imply
\begin{equation}  \label{eq:modello_discreto}
g_{i}= \sum_{s=-q}^{q}h_{s}f_{i-s} +\eta_i, \quad i_1=1, \ldots, n_1, i_2=1,
\ldots, n_2,
\end{equation}
where the range of collected data identifies the so called Field of View (FOV).%

As in \eqref{eq:modello_discreto_base}, we are assuming that all the involved data in (\ref%
{eq:modello_discreto_sistema}) are reshaped in a row-wise
ordering, so that the arising linear system is
\begin{equation}  \label{eq:modello_discreto_sistema}
\widetilde A \tilde f =g -\eta
\end{equation}
where $\widetilde A\in \mathbb{R}^{N(n)\times N(n+2q)}$ is a finite
principal sub-matrix of $\widetilde A_\infty$, with main diagonal containing
$h_{0,0}$, $\tilde f\in \mathbb{R}^{N(n+2q)}$, $g,\eta \in \mathbb{R}^{N(n)}$
and with $N(m)=m_1m_2$, for any two-index $m=(m_1,m_2)$.

According to (\ref{eq:modello_discreto}), the problem is
undetermined since the number of unknowns involved in the
convolution exceeds the number of recorded data. Thus, Boundary
conditions (BCs) are introduced to artificially describe the scene
outside the FOV: the values of unknowns outside the FOV are fixed
or are defined as linear combinations of the unknowns inside the
FOV. In such a way (\ref{eq:modello_discreto_sistema}) is reduced
to a square linear system
\begin{equation}  \label{eq:sistema_quadrato}
A_n f = g -\eta
\end{equation}
with $A_n \in \mathbb{R}^{N(n)\times N(n)}$, $n=(n_1,n_2)$,
$N(n)=n_1 n_2$ and $f,g,\eta \in \mathbb{R}^{N(n)}$. \newline
Though the choice of the BCs does not affect the global spectral
behavior of the matrix, it may have a valuable impact both with
respect to the accuracy of the restored image (especially close to
the boundaries where ringing effects can appear) and to the
computational costs for recovering $f$ from the blurred datum,
with or without noise. Notice also that, typically, the matrix $A$
is very ill-conditioned and there is a significant intersection
between the subspace related to small eigen/singular values and
the high frequency subspace.
\par
The paper is organized as follows. In Section \ref{sez:BC} we
underline the importance of boundary conditions and we summarize
the structural and spectral  properties of matrices arising in the
case of Reflective and Anti-Reflective BCs. Section
\ref{sez:Optimal} is devoted to the presentation of theoretical
results relative to the explicit construction of the optimal
preconditioner for the restoration problem with Anti-Reflective
BCs. In Section \ref{sez:Computational} we report computational
results with respect to two deblurring problems, the former having
a slightly non-symmetric PSF and the latter having an highly
non-symmetric PSF, using the proposed preconditioning (combined
with Tikhonov filtering) for Landweber method. Finally, some
conclusions and perspectives are drown in  Section
\ref{sez:Conclusions}.

\section{The role of boundary conditions in the restoration problem}
\label{sez:BC}

Hereafter we summarize the relevant properties of two recently
proposed type of BCs, i.e., the Reflective \cite{NCT-SISC-1999}
and Anti-Reflective BCs \cite{S-SISC-2003}, with respect both to
structural and spectral properties of the resulting matrices.
Indeed, the use of classical periodic BCs enforces a circulant
structure and the spectral decomposition can be computed
efficiently with the fast Fourier transform (FFT) \cite{D-1979},
but these computational facilities are coupled with significant
ringing effects \cite{BerBoc}, whenever a significant
discontinuity is introduced into the image. Thus, the target is to
obtain the best possible approximation properties, keeping
unaltered the fact that the arising matrix shows an exploitable
structure. Reflective and Anti-Reflective BCs more carefully
describe the scene outside the FOV and give rise to exploitable
structures. Clearly, several other methods deal with this topic in
the  literature, e.g. local mean value \cite{SC-APNUM-} or
extrapolation techniques (see \cite{LB-1991} and references
therein). Nevertheless, the penalty of their good approximation
properties could lie in a linear algebra problem more difficult to
cope with.
%
%
%
Hereafter, as more natural in the applications, we will use a two-index
notation: we denote by $F =\left[ f_{i_1,i_2} \right]_{i_1=1,%
\ldots,n_1,i_2=1,\ldots,n_2}$ the true image inside the FOV and by $G =\left[
g_{i_1,i_2} \right]_{i_1=1,\ldots,n_1,i_2=1,\ldots,n_2}$ the recorded image.

\subsection{Reflective boundary conditions} \label{sez:BC-R}

In \cite{NCT-SISC-1999} Ng \emph{et al.} analyze the use of
Reflective BCs, both from the model and the linear algebra point
of view. The improvement with respect to Periodic BCs amounts to
the preserved continuity of the image. In reality, the scene
outside the FOV is assumed to be a reflection of the scene inside the FOV. For example, with a boundary at $%
x_{1}=0$ and $x_{2}=0$ the reflective condition is given by $f(\pm x_{1},\pm
x_{2})=f(x_{1},x_{2})$.
More precisely, along the borders, the BCs impose
\begin{equation*}
\begin{array}{rclrcll}
f_{i_{1},1-i_{2}} & \!\!=\!\! & f_{i_{1},i_{2}}, & f_{i_{1},n_{2}+i_{2}} &
\!\!=\!\! & f_{i_{1},n_{2}+1-i_{2}}, & \!\!\!\!\text{for any }i_{1}=1,\ldots
,n_{1},\ i_{2}=1,\ldots ,q_{2} \\
f_{1-i_{1},i_{2}} & \!\!=\!\! & f_{i_{1},i_{2}}, & f_{n_{1}+i_{1},i_{2}} &
\!\!=\!\! & f_{n_{1}+1-i_{1},i_{2}}, & \!\!\!\!\text{for any }i_{1}=1,\ldots
,q_{1},\ i_{2}=1,\ldots ,n_{2},%
\end{array}%
\end{equation*}%
and, at the corners, for every $i_{1}=1,\ldots ,q_{1}$, $%
i_{2}=1,\ldots ,q_{2}$ the use of BCs leads to
\begin{equation*}
\begin{array}{rclcrcl}
f_{1-i_{1},1-i_{2}} & = & f_{i_{1},i_{2}}, & \quad  &
f_{n_{1}+i_{1},n_{2}+i_{2}} & = & f_{n_{1}+1-i_{1},n_{2}+1-i_{2}}, \\
f_{1-i_{1},n_{2}+i_{2}} & = & f_{i_{1},n_{2}+1-i_{2}}, & \quad  &
f_{n_{1}+i_{1},1-i_{2}} & = & f_{n_{1}+1-i_{1},i_{2}},%
\end{array}%
\end{equation*}%
i.e., a double reflection, first with respect to one axis and after with
respect to the other, no matter about the order.

As a consequence the rectangular matrix $\widetilde{A}$ is reduced to a
square Toeplitz-plus-Hankel block matrix with Toeplitz-plus-Hankel blocks,
i.e., $A_{n}$ shows the two-level Toeplitz-plus-Hankel structure. Moreover,
if the blurring operator satisfies the strong symmetry condition, i.e., it
is symmetric with respect to each direction, formally
\begin{equation}
h_{|i|}=h_{i}\quad \text{for any }i=-q,\ldots ,q.
\label{eq:strong_simmetry_condition}
\end{equation}%
then the matrix $A_{n}$ belongs to DCT-III matrix algebra and its
spectral decomposition can be computed very efficiently using the
fast discrete cosine transform (DCT-III) \cite{S-1999}. More in
detail, let $\mathcal{C}_{n}=\{A_{n}\in \mathbb{R}^{N(n)\times
N(n)},n=(n_{1},n_{2}),N(n)=n_{1}n_{2}\ |\ A_{n}=R_{n}\Lambda
_{n}R_{n}^{T}\}$ be the two-level DCT-III matrix algebra, i.e.,
the algebra of matrices that are simultaneously diagonalized by
the orthogonal transform
\begin{equation}
R_{n}=R_{n_{1}}\otimes R_{n_{2}},\quad R_{m}=\left[ \sqrt{\frac{2-\delta
_{t,1}}{m}}\cos \left\{ \frac{(s-1)(t-1/2)\pi }{m}\right\} \right]
_{s,t=1}^{m},  \label{Rn}
\end{equation}%
with $\delta _{s,t}$ denoting the Kronecker symbol.\\
The explicit matrix structure  is $A_{n}=\mathrm{Toeplitz}(V)+%
\mathrm{Hankel}(\sigma (V),J\sigma (V))$, with $V=[V_{0}\ V_{1}\
\ldots \ V_{q_{1}}\ 0\ldots 0]$ and where each $V_{i_{1}}$,
$i_{1}=1,\ldots ,q_{1}$ is the unilevel DCT-III matrix associated
to the $i_{1}^{th}$ row of the PSF mask, i.e.,
$V_{i_{1}}=\mathrm{Toeplitz}(v_{i_{1}})+\mathrm{Hankel}(\sigma
(v_{i_{1}}),$ $J\sigma (v_{i_{1}}))$, with
$v_{i_{1}}=[h_{i_{1},0},\ldots ,h_{i_{1},q_{2}},0,\ldots ,0]$.
Here, $\sigma $ denotes the shift operator such that $\sigma
(v_{i_{1}})=[h_{i_{1},1},\ldots ,h_{i_{1},q_{2}},0,\ldots ,0]$ and
$J$ denotes the usual flip matrix; at the block level the same
operations are intended in block-wise sense.

Not only the structural characterization, but also the spectral
description is completely known: let $f$ be the bivariate
generating function associated to the PSF mask
(\ref{eq:coefficienti_psf}), that is
\begin{eqnarray}
f(x_{1},x_{2}) &=&h_{0,0}+2\sum_{s_{1}=1}^{q_{1}}h_{s_{1},0}\cos
(s_{1}x_{1})+2\sum_{s_{2}=1}^{q_{2}}h_{0,s_{2}}\cos (s_{2}x_{2})  \notag \\
&&\ +4\sum_{s_{1}=1}^{q_{1}}\sum_{s_{2}=1}^{q_{2}}h_{s_{1},s_{2}}\cos
(s_{1}x_{1})\cos (s_{2}x_{2}),  \label{eq:funzione_generatrice}
\end{eqnarray}%
then the eigenvalues of the corresponding matrix $A_{n}\in \mathcal{C}_{n}$
are given by
\begin{equation*}
\lambda _{s}(A_{n})=f\left( x_{s_{1}}^{[n_{1}]},x_{s_{2}}^{[n_{2}]}\right)
,\ s=(s_{1},s_{2}),\quad x_{r}^{[m]}=\frac{(r-1)\pi }{m},
\end{equation*}%
where $s_{1}=1,\ldots ,n_{1}$, $s_{2}=1,\ldots ,n_{2}$, and where
the two-index notation highlights the tensorial structure of the
corresponding eigenvectors. Finally, note that standard operations
like matrix-vector products, resolution of linear systems and
eigenvalues evaluations can be performed by means of FCT-III
\cite{NCT-SISC-1999} within $O(n_{1}n_{2}\log (n_{1}n_{2}))$
arithmetic operations (ops).

\subsection{Anti-Reflective boundary conditions}

\label{sez:BC-AR} More recently, Anti-Reflective boundary
conditions
(AR-BCs) have been proposed in \cite{S-SISC-2003} and studied \cite%
{ADS-AR1,ADS-AR2,ADNS-AR3,DENPS-SPIE-2003,DS-IP-2005,Anti-Reflective3,P-NLA-2006,TP-2010}%
. The improvement relies in the fact that not only the continuity
of the image, but also of the normal derivative, are guaranteed at
the boundary. This higher regularity,  not shared with Dirichlet
or periodic BCs, and only partially shared with reflective BCs,
significantly reduces typical ringing artifacts in the restored
image.

The key idea is simply to assume that the scene outside the FOV is
the anti-reflection of the scene inside the FOV. For example, with
a boundary at $x_{1}=0$ the anti-reflective condition imposes
$f(-x_{1},x_{2})-f(x_{1}^{\ast
},x_{2})=-(f(x_{1},x_{2})-f(x_{1}^{\ast },x_{2}))$, for any
$x_{2},$ where $x_{1}^{\ast }$ is the center of the
one-dimensional anti-reflection, i.e.,
\begin{equation*}
f(-x_{1},x_{2})=2f(x_{1}^{\ast },x_{2})-f(x_{1},x_{2}),\text{for any }x_{2}.
\end{equation*}%
Notice that, in order to preserve a tensorial structure,  a double
anti-reflection, first with respect to one axis and after with
respect to the other, is considered at the
corners, 
so that the BCs impose
\begin{equation*}
f(-x_{1},-x_{2})=4f(x_{1}^{\ast },x_{2}^{\ast })-2f(x_{1}^{\ast
},x_{2})-2f(x_{1},x_{2}^{\ast })+f(x_{1},x_{2}),%
\end{equation*}%
where $(x_{1}^{\ast },x_{2}^{\ast })$ is the center of the
two-dimensional anti-reflection. More specifically, by choosing as
center of the anti-reflection the first available data, along the
borders, the BCs impose
\begin{equation*}
\begin{array}{lll}
\!\!f_{1-i_{1},i_{2}}\!\!\!=\!2f_{1,i_{2}}\!-\!f_{i_{1}+1,i_{2}},\  &
\!\!\!f_{n_{1}+i_{1},i_{2}}\!\!\!=\!2f_{n_{1},i_{2}}\!-%
\!f_{n_{1}-i_{1},i_{2}}, & \!\!\!i_{1}=1,\ldots ,q_{1},\ \!i_{2}=1,\ldots
,n_{2}, \\
\!\!f_{i_{1},1-i_{2}}\!\!\!=\!2f_{i_{1},1}\!-\!f_{i_{1},i_{2}+1},\  &
\!\!\!f_{i_{1},n_{2}+i_{2}}\!\!\!=\!2f_{i_{1},n_{2}}\!-%
\!f_{i_{1},n_{2}-i_{2}}, & \!\!\!i_{1}=1,\ldots ,n_{1},\ \!i_{2}=1,\ldots
,q_{2}.%
\end{array}%
\end{equation*}%
At the corners, for any $i_{1}=1,\ldots ,q_{1}$ and $%
i_{2}=1,\ldots ,q_{2}$, we find
\begin{equation*}
\begin{array}{rcl}
f_{1-i_{1},1-i_{2}} & = &
4f_{1,1}-2f_{1,i_{2}+1}-2f_{i_{1}+1,1}+f_{i_{1}+1,i_{2}+1}, \\
f_{1-i_{1},n_{2}+i_{2}} & = &
4f_{1,n_{2}}-2f_{1,n_{2}-i_{2}}-2f_{i_{1}+1,n_{2}}+f_{i_{1}+1,n_{2}-i_{2}},
\\
f_{n_{1}+i_{1},1-i_{2}} & = &
4f_{n_{1},1}-2f_{n_{1},i_{2}+1}-2f_{n_{1}-i_{1},1}+f_{n_{1}-i_{1},i_{2}+1},
\\
f_{n_{1}+i_{1},n_{2}+i_{2}} & = &
4f_{n_{1},n_{2}}-2f_{n_{1},n_{1}-i_{2}}-2f_{n_{1}-i_{1},n_{2}}+f_{n_{1}-i_{1},n_{2}-i_{2}}.%
\end{array}%
\end{equation*}%
As a matter of fact, the rectangular matrix $\widetilde{A}$ is
reduced to a square Toeplitz-plus-Hankel block matrix with
Toeplitz-plus-Hankel blocks, plus an additional structured low
rank matrix. More details on this structure in the general case
are reported in Section \ref{sez:Optimal}. Hereafter, we observe
that again under the assumption of strong symmetry of
the PSF and of a mild finite support condition (more precisely $h_{i}=0$ if $%
|i_{j}|\geq n-2$, 
for some $j\in \{1,2\}$), the linear system $A_{n}f=g$ is such
that $A_{n}$ belongs to the $\mathcal{AR}_{n}^{2D}$ commutative
matrix
algebra \cite{ADS-AR2}. \par
This new algebra shares some properties with the $%
\tau $ (or DST-I) algebra \cite{BC-1983}. Going inside the
definition, a matrix $A_{n}\in \mathcal{AR}_{n}^{2D}$ has
the following block structure 
\begin{equation}
A_{n}=\left[
\begin{array}{c|c|c}
\tilde{H}_{0}+Z_{1} & 0^{T} & 0 \\ \hline
\tilde{H}_{1}+Z_{2} &  & 0 \\
\vdots  &  & \vdots  \\
\tilde{H}_{q_{1}-1}+Z_{q_{1}} &  & 0 \\
\tilde{H}_{q_{1}} & \ \tau (\tilde{H}_{0},\ldots ,\tilde{H}_{q_{1}})\  &
\tilde{H}_{q_{1}} \\
0 &  & \tilde{H}_{q_{1}-1}+Z_{q_{1}} \\
\vdots  &  & \vdots  \\
0 &  & \tilde{H}_{1}+Z_{2} \\ \hline
0 & 0^{T} & \tilde{H}_{0}+Z_{1}%
\end{array}%
\right] ,  \label{eq:AR2D}
\end{equation}%
where $\tau (\tilde{H}_{0},\ldots ,\tilde{H}_{q_{1}})$ is a block $\tau $
matrix with respect to the $\mathcal{AR}^{1D}$ blocks $\tilde{H}_{i_{1}}$, $%
i_{1}=1,\ldots ,q_{1}$ and $Z_{k}=2\sum_{t=k}^{q_{1}}\tilde{H}_{t}$ for $%
k=1,\ldots ,q_{1}$. In particular, the $\mathcal{AR}^{1D}$ block $\tilde{H}%
_{i_{1}}$ is associated to $i_{1}^{th}$ row of the PSF, i.e., $%
h_{i_{1}}^{[1D]}=[h_{i_{1},i_{2}}]_{i_{2}=-q_{2},\ldots ,q_{2}}$ and it is
defined as
\begin{equation}
\tilde{H}_{i_{1}}=\left[
\begin{array}{c|c|c}
h_{i_{1},0}+z_{i_{1},1} & 0^{T} & 0 \\ \hline
h_{i_{1},1}+z_{i_{1},2} &  & 0 \\
\vdots  &  & \vdots  \\
h_{i_{1},q_{2}-1}+z_{i_{1},q_{2}} &  & 0 \\
h_{i_{1},q_{2}} & \ \tau (h_{i_{1},0},\ldots ,h_{i_{1},q_{2}})\  &
h_{i_{1},q_{2}} \\
0 &  & h_{i_{1},q_{2}-1}+z_{i_{1},q_{2}} \\
\vdots  &  & \vdots  \\
0 &  & h_{i_{1},1}+z_{i_{1},2} \\ \hline
0 & 0^{T} & h_{i_{1},0}+z_{i_{1},1}%
\end{array}%
\right] ,  \label{eq:AR2Dbis}
\end{equation}%
where $z_{i_{1},k}=2\sum_{t=k}^{q_{2}}h_{i_{1},t}$ for $k=1,\ldots ,q_{2}$
and $\tau (h_{i_{1},0},\ldots ,h_{i_{1},q_{2}})$ is the unilevel $\tau $
matrix associated to the one-dimensional PSF $h_{i_{1}}^{[1D]}$ previously defined.
Notice that the rank-1 correction given by the elements $z_{i_{1},k}$
pertains to the contribution of the anti-reflection centers with respect to
the vertical borders, while the low rank correction given by the matrices $%
Z_{k}$ pertains to the contribution of the anti-reflection centers with
respect to the horizontal borders.

Favorable computational properties are guaranteed also by virtue
of the $\tau $ structure, so that, firstly we briefly summarize
the relevant properties of the two-level $\tau$ algebra
\cite{BC-1983}.
Let $\mathcal{T}_{n}=\{A_{n}\in \mathbb{R}^{N(n)\times
N(n)},n=(n_{1},n_{2}),N(n)=n_{1}n_{2}\ |\ A_{n}=Q_{n}\Lambda _{n}Q_{n}\}$ be
the two-level $\tau $ matrix algebra, i.e., the algebra of matrices that are
simultaneously diagonalized by the symmetric orthogonal transform
\begin{equation}
Q_{n}=Q_{n_{1}}\otimes Q_{n_{2}},\quad Q_{m}=\left[ \sqrt{\frac{2}{m+1}}\sin
\left\{ \frac{st\pi }{m+1}\right\} \right] _{s,t=1}^{m}.  \label{Qn}
\end{equation}%
With the same notation as the DCT-III algebra case, the explicit structure
of the matrix is two level Toeplitz-plus-Hankel. More precisely,
$A_{n}=\mathrm{Toeplitz}(V)-\mathrm{Hankel}(\sigma ^{2}(V),$
$J\sigma ^{2}(V))$
with $V=[V_{0}\ V_{1}\ \ldots \ V_{q_{1}}\ 0\ldots 0]$, where each $V_{i_{1}}
$, $i_{1}=1,\ldots ,q_{1}$ is a the unilevel $\tau $ matrix associated to
the $i_{1}^{th}$ row of the PSF mask, i.e., $V_{i_{1}}\!=\!\mathrm{Toeplitz}%
(v_{i_{1}})-\mathrm{Hankel}(\sigma ^{2}(v_{i_{1}}),J\sigma ^{2}(v_{i_{1}}))$
with $v_{i_{1}}=[h_{i_{1},0},\ldots ,h_{i_{1},q_{2}},0,\ldots ,0]$. Here, we
denote by $\sigma ^{2}$ the double shift operator such that $\sigma
^{2}(v_{i_{1}})=[h_{i_{1},2},\ldots ,h_{i_{1},q_{2}},0,\ldots ,0]$; at the
block level the same operations are intended in block-wise sense.
The spectral characterization is also completely known since for any $%
A_{n}\in \mathcal{T}_{n}$ the related eigenvalues are given by
\begin{equation*}
\lambda _{s}(A_{n})=f\left( x_{s_{1}}^{[n_{1}]},x_{s_{2}}^{[n_{2}]}\right)
,s=(s_{1},s_{2}),\quad x_{r}^{[m]}=\frac{r\pi }{m+1},
\end{equation*}%
where $s_{1}=1,\ldots ,n_{1}$, $s_{2}=1,\ldots ,n_{2}$, and $f$ is the
bivariate generating function associated to the $PSF$ defined in (\ref%
{eq:funzione_generatrice}). 

As in the DCT-III case, standard operations like matrix-vector products,
resolution of linear systems and eigenvalues evaluations can be performed by
means of FST-I within $O(n_{1}n_{2}\log (n_{1}n_{2}))$ (ops).
Now, with respect to the $\mathcal{AR}_{n}^{2D}$ matrix algebra, a
complete spectral characterization is given in
\cite{ADS-AR2,ADNS-AR3}. Of considerable importance is the
existence of a transform $T_{n}$ that simultaneously diagonalizes
all the matrices belonging to $\mathcal{AR}_{n}^{2D}$, although
the orthogonality property is partially lost.

\begin{theorem}
{\rm \textrm{\cite{ADNS-AR3}}} \label{teo:teo_Jordan} Any matrix $A_n\in \mathcal{AR}_n^{2D}$, $%
n=(n_1,n_2)$, can be diagonalized by $T_n$, i.e.,
\begin{equation*}
A_n=T_n \Lambda_n \widetilde T_n,\quad \widetilde{T}_n= T_n^{-1}
\end{equation*}
where $T_n=T_{n_1}\otimes T_{n_2}$, $\widetilde T_n= \widetilde
T_{n_1}\otimes \widetilde T_{n_2}$, with
\begin{equation*}
T_m=\left [
\begin{array}{ccc}
\alpha_m^{-1} & 0^T & 0 \\
\  &  &  \\
\alpha_m^{-1}p & Q_{m-2} & \alpha_m^{-1}Jp \\
\  &  &  \\
0 & 0^T & \alpha_m^{-1}%
\end{array}
\right ] \quad and \quad \widetilde{T}_m=\left [
\begin{array}{ccc}
\alpha_m & 0^T & 0 \\
\  &  &  \\
-Q_{m-2}p & Q_{m-2} & -Q_{m-2}Jp \\
\  &  &  \\
0 & 0^T & \alpha_m%
\end{array}
\right ]
\end{equation*}
The entries of the vector $p\in\mathbb{R}^{m-2}$ are defined as $p_j=1-{j}/{%
(m-1)}$, $j=1,\ldots,m-2$, $J\in\mathbb{R}^{m-2\times m-2}$ is the flip
matrix, and $\alpha_m$ is a normalizing factor chosen such that the
Euclidean norm of the first and last column of $T_m$ will be equal to $1$.
\end{theorem}

\begin{theorem}
{\rm \textrm{\cite{ADS-AR2}}} \label{teo:teo_autovalori} Let $A_n\in \mathcal{AR}%
_n^{2D}$, $n=(n_1,n_2)$, the matrix related to the PSF $%
h_{PSF}=[h_{i_1,i_2}]_{i_1=-q_1,\ldots,q_1, i_2=-q_2,\ldots,q_2}$. Then, the
eigenvalues of $A_n$ are given by

\begin{itemize}
\item $1$ with algebraic multiplicity $4$,

\item the $n_2-2$ eigenvalues of the unilevel $\tau$ matrix related to the
one-dimensional PSF $h^{\{r\}}=[\sum_{i_1=-q_1}^{q_1}h_{i_1,-q_2}, \ldots,
\sum_{i_1=-q_1}^{q_1}h_{i_1,q_2}]$,
each one with algebraic multiplicity $2$,

\item the $n_1-2$ eigenvalues of the unilevel $\tau$ matrix related to the
one-dimensional PSF $h^{\{c\}}= [\sum_{i_2=-q_2}^{q_2}h_{-q_1,i_2}, \ldots,
\sum_{i_2=-q_2}^{q_2}h_{q_1,i_2}]$,
each one with algebraic multiplicity $2$,

\item the $(n_1-2)(n_2-2)$ eigenvalues of the two-level $\tau$ matrix
related to the two-dimensional PSF $h_{PSF}$.
\end{itemize}
\end{theorem}

It's worthwhile noticing that the three sets of multiple
eigenvalues are  related to the type of low rank correction
imposed by the BCs through the centers of
the anti-reflections. More precisely, the eigenvalues of $%
\tau_{n_2-2}(h^{\{r\}})$ and of $\tau_{n_1-2}(h^{\{c\}})$ take
into account the condensed PSF information considered along the
horizontal and vertical borders respectively, while the eigenvalue
equal to $1$ takes into account the condensed information of the
whole PSF at the four corners. In addition,  the spectral
characterization can be
completely described again in terms of the generating function associated to the $%
PSF$ defined in (\ref{eq:funzione_generatrice}), simply by extending to $0$
the standard $\tau$ evaluation grid, i.e., it holds
\begin{equation*}
\lambda_{s}(A_n)= f\left(x_{s_1}^{[n_1]},x_{s_2}^{[n_2]}\right),
s=(s_1,s_2), s_j=0,\ldots,n_j, \quad x_r^{[m]}=\frac{r \pi}{m+1},
\end{equation*}
where the $0-$index refers to the first/last columns of the matrix
$T_m$ \cite{ADS-AR2}. See \cite{ADS-AR1,ADNS-AR3} for some
algorithms related to standard operations like matrix-vector
products, resolution of linear systems and eigenvalues evaluations
with a computational cost of $O(n_1n_2\log (n_1 n_2)) $ ops. In
fact, the computational cost of the inverse transform is
comparable with the direct transform one and the very true penalty
seems to be the loss of orthogonality due to the first/last column
of the matrix $T_m$.

We stress that the latter complete spectral characterization and
the related fast algorithms for computing the eigenvalues are
essential for the fast implementation of the regularization
algorithms used in the numerical section.


\section{Optimal preconditioning}
\label{sez:Optimal}

In this section we consider in more detail the matrices arising
when Anti-Reflective BCs are applied in the case of a
non-symmetric PSF, the aim being to define the corresponding
optimal preconditioner in the $\mathcal{AR}_n^{2D}$ algebra.
\newline More precisely, let $A=A(h)$ be the Anti-Reflective matrix
generated by the generic PSF
$h_{PSF}=\left[h_{i_1,i_2}\right]_{i_1=-q_1,\ldots,q_1,i_2=-q_2,\ldots,q_2}$
and let $P=P(s)\in \mathcal{AR}_n^{2D}$ be the Anti-Reflective
matrix generated by the symmetrized PSF $s_{PSF}=\left[s_{i_1,i_2}\right]%
_{i_1=-q_1,\ldots,q_1,i_2=-q_2,\ldots,q_2}$. We are looking for the optimal
preconditioner $P^*=P^*(s^*)$ in the sense that
\begin{equation}  \label{eq:Pstar_optimal}
P^*=\underset{P\in \mathcal{AR}_n^{2D}}{\arg }\min \left\Vert A-P\right\Vert
_{\mathcal{F}}^{2}\text{ , \ }\bar{s}=\underset{s}{\arg }\min \left\Vert
A(f)-P(s)\right\Vert _{\mathcal{F}}^{2}\text{ ,}
\end{equation}%
where $\left\Vert \cdot \right\Vert _{\mathcal{F}}$ is the Frobenius norm,
defined as $\left\Vert A\right\Vert _{\mathcal{F}}=\sqrt{\underset{i,j}{\sum
}\left\vert a_{i,j}\right\vert ^{2}}$.
Indeed, an analogous result is know in \cite{NCT-SISC-1999} with respect to
Reflective BCs: given a generic PSF $h_{PSF}=\left[h_{i_1,i_2}\right]
$, the optimal preconditioner in the DCT-III matrix algebra is generated by
the strongly symmetric PSF $s_{PSF}=\left[s_{i_1,i_2}\right]
$, given by
\begin{equation}  \label{eq:simmetrizzata2d}
s_{\pm i_1,\pm i_2}=\dfrac{%
h_{-i_1,-i_2}+h_{-i_1,i_2}+h_{i_1,-i_2}+h_{i_1,i_2}}{4}\text{.}
\end{equation}%
Our interest is clearly motivated by the computational facilities proper of $%
\mathcal{AR}_n^{2D}$ algebra, coupled with its better approximation
properties.
We preliminary consider the one-dimensional case in order to introduce the
key idea in the proof with a simpler notation. Moreover, the proof argument
of the two-dimensional case is also strongly connected to the
one-dimensional one.

\subsection{One-dimensional case}

Let us consider a generic PSF $h_{PSF}=\left[ h_{i}\right] _{i=-q,\ldots ,q}$%
. As introduced in Section \ref{sez:BC-AR}, the idea is to apply an
anti-reflection with respect to the border points $f_{1}$ and $f_{n}$. Thus,
we impose%
\begin{equation*}
f_{1-i}=2f_{1}-f_{1+i}\text{,}\quad f_{n+i}=2f_{n}-f_{n-i}\text{,}\quad
i=1,\ldots ,q\text{.}
\end{equation*}%
The resulting matrix shows a more involved structure with respect to the
Reflective BCs, i.e., it is Toeplitz + Hankel plus a structured low rank
correction matrix, as follows 
\begin{equation}
A=\left[
\begin{tabular}{c|ccccc|c}
$v_{0}$ & \multicolumn{5}{|c|}{$u^{T}$} & $0$ \\ \hline
$v_{1}$ &  &  &  &  &  &  \\
$\vdots $ &  &  &  &  &  &  \\
$v_{q}$ &  &  & $B$ &  &  & $w_{q}$ \\
&  &  &  &  &  & $\vdots $ \\
&  &  &  &  &  & $w_{1}$ \\ \hline
$0$ & \multicolumn{5}{|c|}{$-(Ju)^{T}$} & $w_{0}$%
\end{tabular}%
\right]   \label{eq:ARnosym1D}
\end{equation}%
with%
\begin{eqnarray*}
u^{T} &=&\left[ h_{-1}-h_{1},\ldots ,h_{-q}-h_{q},0,\ldots ,0\right] \text{,}%
\quad -(Ju)^{T}=\left[ 0,\ldots ,0,h_{q}-h_{-q},\ldots ,h_{1}-h_{-1}\right]
\text{,} \\
v_{k} &=&h_{k}+2\underset{j=k+1}{\overset{q}{\sum }}h_{j}\text{,}\quad
w_{k}=h_{-k}+2\underset{j=k+1}{\overset{q}{\sum }}h_{-j}\text{,} \\
B &=&T(\left[ h_{-q},\ldots ,h_{q}\right] )-H_{\mathrm{TL}}(\left[
h_{2},\ldots ,h_{q}\right] )-H_{\mathrm{BR}}(\left[ h_{-2},\ldots ,h_{-q}%
\right] )\text{,}
\end{eqnarray*}%
where $T(\left[ h_{-q},\ldots ,h_{q}\right] )$ is the Toeplitz matrix
associated to the PSF $h_{PSF}$, 
while $H_{\mathrm{TL}}(\left[ h_{2},\right.$ $\left. \ldots ,h_{q}\right] )$ and $H_{\mathrm{%
BR}}(\left[ h_{-2},\ldots ,h_{-q}\right] )$ are respectively the top-left
Hankel and the bottom-right Hankel matrices
\begin{equation*}
{
\begin{tabular}{cc}
$H_{\mathrm{TL}}=\left[
\begin{array}{ccccccc}
h_{2} & h_{3} & \cdots  & h_{q} & 0 & \cdots  & 0 \\
h_{3} &  & h_{q} & 0 &  &  & \vdots  \\
\vdots  & h_{q} & 0 &  &  &  &  \\
h_{q} & 0 &  &  &  &  &  \\
0 &  &  &  &  &  &  \\
\vdots  &  &  &  &  &  & \vdots  \\
0 & \cdots  &  &  &  & \ldots  & 0%
\end{array}%
\right] \text{,}$  & $H_{\mathrm{BR}}=\left[
\begin{array}{ccccccc}
0 & \cdots  &  &  &  & \cdots  & 0 \\
\vdots  &  &  &  &  &  & \vdots  \\
&  &  &  &  &  & 0 \\
&  &  &  &  & 0 & h_{-q} \\
&  &  &  & 0 & h_{-q} & \vdots  \\
\vdots  &  &  & 0 & h_{-q} &  & h_{-3} \\
0 & \cdots  & 0 & h_{-q} & \cdots  & h_{-3} & h_{-2}%
\end{array}%
\right] \text{.}$
\end{tabular}%
}
\end{equation*}%
On the other hand, the Anti-Reflective matrix $P\in \mathcal{AR}^{1D}$
generated by a strongly symmetric PSF $s_{PSF}=\left[ s_{q},\ldots
,s_{1},s_{0},s_{1},\ldots ,s_{q}\right] $, among which the minimizer $%
P^{\ast }$ in (\ref{eq:Pstar_optimal}) will be searched, is clearly given by
\begin{equation*}
P=\left[
\begin{tabular}{c|ccccc|c}
$r_{0}$ & \multicolumn{5}{|c|}{$0^{T}$} & $0$ \\ \hline
$r_{1}$ &  &  &  &  &  &  \\
$\vdots $ &  &  &  &  &  &  \\
$r_{q}$ &  &  & $\tau (s)$ &  &  & $r_{q}$ \\
&  &  &  &  &  & $\vdots $ \\
&  &  &  &  &  & $r_{1}$ \\ \hline
$0$ & \multicolumn{5}{|c|}{$0^{T}$} & $r_{0}$%
\end{tabular}%
\right]
\end{equation*}%
where $r_{k}=s_{k}+2\underset{j=k+1}{\overset{q}{\sum }}s_{j}$ and $\tau (s)$
is the $\tau $ (or DST-I) matrix generated by the PSF $s_{PSF}$.

The optimality of the Anti-Reflective matrix generated by the symmetrized
PSF defined as
\begin{equation}
s_{\pm i}=\dfrac{h_{-i}+h_{i}}{2}\text{.}  \label{eq:simmetrizzata1d}
\end{equation}%
can be proved analogously as in \cite{NCT-SISC-1999} with respect to the
internal part $C_{I}=B-\tau (s)$ and by invoking a non-overlapping splitting
argument in order to deal with the external border $C_{B}$. 
In fact, we have
\begin{equation*}
\left\Vert C\right\Vert _{\mathcal{F}}^{2}=\left\Vert C_{I}\right\Vert _{%
\mathcal{F}}^{2}+\left\Vert C_{B}\right\Vert _{\mathcal{F}}^{2}
\end{equation*}%
and it easy to show that the minimizer found for the first term is the same
than for the second one.

Notice that $\Vert u^{T}\Vert ^{2}$ and $\Vert -(Ju)^{T}\Vert ^{2}$ are
constant terms in the minimization process. So, as naturally expected, the
obtained minimum value will be greater, the greater is the loss of symmetry
in the PSF. Moreover, with the choice (\ref{eq:simmetrizzata1d}), the first
and last column in $C_{B}$ share the same norm, i.e., again the most
favourable situation.
It is worth stressing that the minimization process of the second term $C_{B}
$ allows to highlight as the tuning of each minimization parameter can be
performed just by considering two proper corresponding entries in the
matrix, i.e.,
\begin{equation*}
(r_{p}-v_{p})+(r_{p}-w_{p})=0,\quad p=0,\ldots ,q
\end{equation*}%
where $v_{p}$ and $w_{p}$ are linear combination of the same coefficients
with positive and negative indices, respectively.
Taking this fact in mind, we can now consider a more geometrical approach to
the proof, that allows to greatly simplify also the proof with respect to
the minimization of the internal part and can be applied to any type of BCs
based on the fact that the values of unknowns outside the FOV are fixed or
are defined as linear combinations of the unknowns inside the FOV.

%

\begin{theorem}
Let $A=A(h)$ be the Anti-Reflective matrix generated by the generic PSF $%
h_{PSF}=\left[h_{i}\right]_{i=-q,\ldots,q}$.  The optimal
preconditioner in the $\mathcal{AR}_n^{1D}$ algebra is the matrix
associated with the symmetrized PSF $s_{PSF}=\left[s_q, \ldots,
s_1, s_0, s_1, \ldots, s_q\right] $, with
\begin{equation}  \label{eq:simmetrizzata1d-bis}
s_{i}=\dfrac{h_{-i}+h_{i}}{2}\text{.}
\end{equation}
\end{theorem}

\begin{proof}
\begin{figure}[tb]
\centering
\includegraphics[scale=0.3]{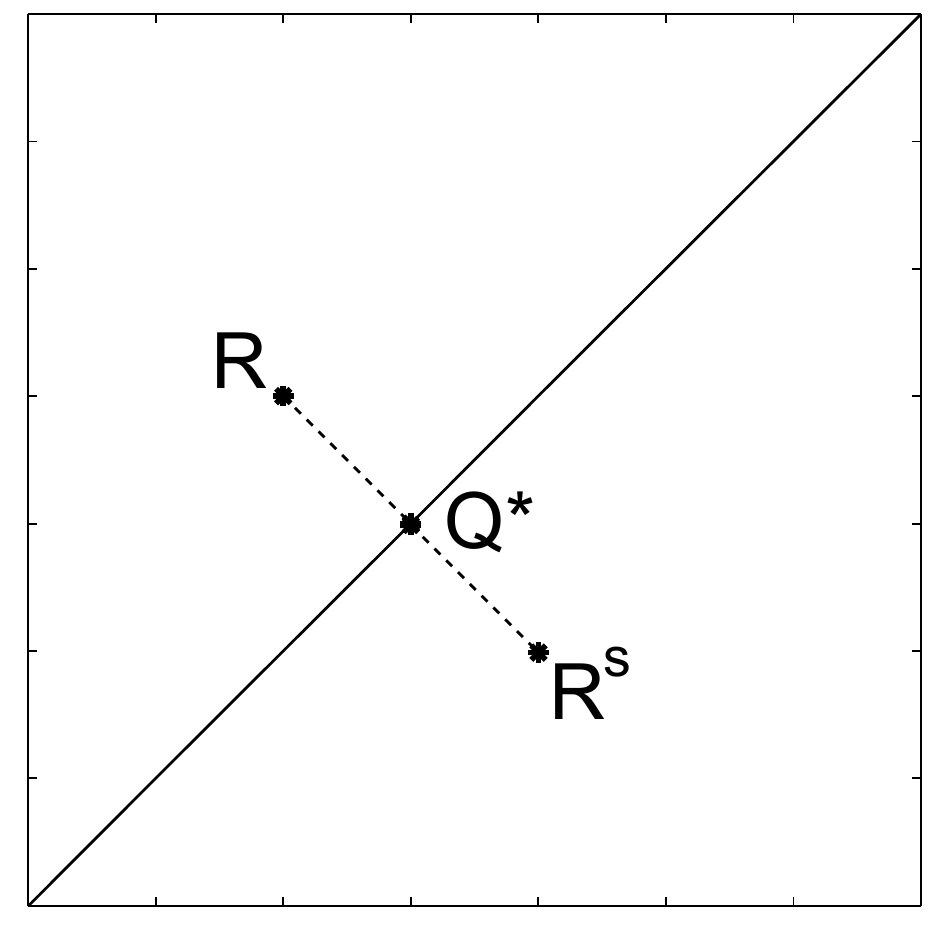}
\caption{A point $R$, its swapped point $R^S$, the optimal approximation of both $Q^*$.}
\label{fig:retta_proiezione}
\end{figure}
Preliminarily, as shown in Figure \ref{fig:retta_proiezione}, we
simply observe that if we consider in the Cartesian plane a point
$R=(R_x,R_y)$,
its optimal approximation $Q^*$, among the points $Q=(Q_x,Q_y)$ such that $%
Q_x=Q_y$, is obtained as the intersection between the line $y=x$,
with the perpendicular line that pass through $R$, that is
\begin{equation*}
\left\{
\begin{array}{l}
y-R_y=-(x-R_x) \\
y=x%
\end{array}%
\right.
\end{equation*}%
hence $Q_x^*=Q_y^*=\left(R_x+R_y\right) /2$. The same holds true if we
consider the swapped point $R^S=(R_y,R_x)$, since they share the same
distance, i.e., $d(R,Q^*)=d(R^S,Q^*)$. %

Clearly, due to linearity of obtained expression, this result can be extended also in the case of any linear combination of coordinates.
Thus, by explicitly exploiting the structure of $A$ and $P$, we define as
x-coordinate of a point the entry with negative index and as y-coordinate of
the same point the corresponding entry with positive index. For the sake of
simplicity we report an example for $q=3$, in which we put in evidence the x
or y coordinate definition,
\begin{eqnarray*}
C &=& A - P = \left[
\begin{array}{ccccccc}
\omega_{0}^y & \nu _{1}^x & \nu _{2}^x & \nu _{3}^x &  &  &  \\
\omega _{1}^y & \zeta _{0}^y & \zeta _{2}^x & \theta _{2}^x & \theta_{3}^x &
&  \\
\omega_{2}^y & \zeta_{1}^y & \theta_{0} & \theta_{1}^x & \theta_{2}^x &
\theta_{3}^x &  \\
\omega_{3}^y & \theta_{2}^y & \theta_{1}^y & \theta_{0} & \theta_{1}^x &
\theta_{2}^x & \omega_{3}^x \\
& \theta_{3}^y & \theta_{2}^y & \theta _{1}^y & \theta_{0} & \zeta_{1}^x &
\omega_{2}^x \\
&  & \theta _{3}^y & \theta_{2}^y & \zeta_{2}^y & \zeta_{0}^x & \omega _{1}^x
\\
&  &  & \nu_{3}^y & \nu _{2}^y & \nu_{1}^y & \omega_{0}^x%
\end{array}%
\right] -\left[
\begin{array}{ccccccc}
\hat{\omega}_{0}^y & 0 & 0 & 0 &  &  &  \\
\hat{\omega}_{1}^y & \hat{\zeta}_{0}^y & \hat{\zeta}_{2}^x & \hat{\theta}%
_{2}^x & \hat{\theta}_{3}^x &  &  \\
\hat{\omega}_{2}^y & \hat{\zeta}_{1}^y & \hat{\theta}_{0} & \hat{\theta}%
_{1}^x & \hat{\theta}_{2}^x & \hat{\theta}_{3}^x &  \\
\hat{\omega}_{3}^y & \hat{\theta}_{2}^y & \hat{\theta}_{1}^y & \hat{\theta}%
_{0} & \hat{\theta}_{1}^x & \hat{\theta}_{2}^x & \hat{\omega}_{3}^x \\
& \hat{\theta}_{3}^y & \hat{\theta}_{2}^y & \hat{\theta}_{1}^y & \hat{\theta}%
_{0} & \hat{\zeta}_{1}^x & \hat{\omega}_{2}^x \\
&  & \hat{\theta}_{3}^y & \hat{\theta}_{2}^y & \hat{\zeta}_{2}^y & \hat{\zeta%
}_{0}^x & \hat{\omega}_{1}^x \\
&  &  & 0 & 0 & 0 & \hat{\omega}_{0}^x%
\end{array}%
\right] \text{,}
\end{eqnarray*}%
Here, we set the points
\begin{eqnarray*}
\Theta_i &=& (\theta_i^x,\theta_i^y)=(h_{-i},h_{i}) \\
\Omega_i &=& (\omega_i^x,\omega_i^y)=(h_{-i}+2
\sum_{j=i+1}^{q}h_{-j},h_{i}+2 \sum_{j=i+1}^{q}h_{j})=(\theta_{i}^x+2
\sum_{j=i+1}^{q}\theta_j^x,\theta_{i}^y+2 \sum_{j=i+1}^{q}\theta_j^y) \\
N_i &=&(\nu_i^x,\nu_i^y)=(h_{-i}-h_{i},h_{i}-h_{-i})
=(\theta_i^x-\theta_i^{Sx},\theta_i^y-\theta_i^{Sy})
\end{eqnarray*}
and
\begin{eqnarray*}
Z_0 &=&
(\zeta_0^x,\zeta_0^y)=(h_{0}-h_{-2},h_{0}-h_{2})=(\theta_{0}^x-\theta_{2}^x,%
\theta_{0}^y-\theta_{2}^y) \\
Z_1&=&
(\zeta_1^x,\zeta_1^y)=(h_{-1}-h_{-3},h_{1}-h_{3})=(\theta_{1}^x-%
\theta_{3}^x,\theta_{1}^y-\theta_{3}^y) \\
Z_2 &=&
(\zeta_2^x,\zeta_2^y)=(h_{-1}-h_{3},h_{1}-h_{-3})=(\theta_{1}^x-%
\theta_{3}^{Sx},\theta_{1}^y-\theta_{3}^{Sy})
\end{eqnarray*}
related to the Hankel corrections. The points $\hat \Theta_i$, $\hat \Omega_i
$, $\hat Z_i$ related to the matrix $P$ are defined in a similar manner,
taking into account the strong symmetry property, i.e. they have the same x
and y coordinates. %
More in general, the key idea is to transform the original minimization
problem in the equivalent problem of minimizing the quantity
\begin{equation}
\begin{array}{l}
c_{0}\mathbf{d}(\Theta _{0},\hat{\Theta}_{0})^{2}+\ldots +c_{q}\mathbf{d}%
(\Theta_{q},\hat{\Theta}_{q})^{2}+\mathbf{d}(Z_{0},\hat{Z}_{0})^{2}+\ldots +%
\mathbf{d}(Z_{m},\hat{Z}_{m})^{2} \quad \quad \\
\quad +\mathbf{d}(\Omega _{0},\hat{\Omega}_{0})^{2}+\ldots +\mathbf{d}%
(\Omega _{q},\hat{\Omega}_{q})^{2}+\mathbf{d}(N_{1},0)^{2}+\ldots +\mathbf{d}%
(N_{q},0)^{2},%
\end{array}
\label{eq:points}
\end{equation}%
where $c_{j}$ are some constants taking into account the number of constant
Toeplitz entries. Now, by referring to the initial geometrical observation,
we start from points pertaining to the Toeplitz part, that can be minimized
separately, and we obtain the minimizer (\ref{eq:simmetrizzata1d-bis}). It
is also an easy check to prove the same claim with respect to any other
terms, by invoking the quoted linearity argument. %
\end{proof}


\subsection{Two-dimensional case}

Let  $h_{PSF}=\left[h_{i_1,i_2}\right]%
_{i_1=-q_1,\ldots,q_1,i_2=-q_2,\ldots,q_2}$ be a generic PSF. As introduced in Section \ref%
{sez:BC-AR}, the idea is to apply an anti-reflection with respect to the
border points $f_{1,i_2}$, $f_{i_1,1}$ and $f_{n_1,i_2}$, $f_{i_1,n_2}$, $%
i_1=1,\ldots,n_1$, $i_2=1,\ldots, n_2$, and a double
anti-reflection at the corners in order to preserve the tensorial
structure. The resulting matrix shows a more involved structure,
i.e., it is block Toeplitz + Hankel with Toeplitz + Hankel blocks
plus a structured low rank
correction matrix, as follows %
\begin{equation}
A=\left[
\begin{tabular}{c|ccccc|c}
$V_{0}$ & \multicolumn{5}{|c|}{$U$} & $0$ \\ \hline
$V_{1}$ &  &  &  &  &  &  \\
$\vdots $ &  &  &  &  &  &  \\
$V_{q_1}$ &  &  & $B$ &  &  & $W_{q_1}$ \\
&  &  &  &  &  & $\vdots $ \\
&  &  &  &  &  & $W_{1}$ \\ \hline
$0$ & \multicolumn{5}{|c|}{$-JU$} & $W_{0}$%
\end{tabular}%
\right] \text{,}  \label{eq:2d_form}
\end{equation}%
with%
\begin{eqnarray*}
U &=&\left[ \hat H_{-1}-\hat H_{1},\ldots ,\hat H_{-q_1}-\hat
H_{q_1},0,\ldots ,0\right] \text{,} \ -JU =\left[ 0,\ldots ,0,\hat
H_{q_1}-\hat H_{-q_1},\ldots ,\hat H_{1}-\hat H_{-1}\right] \text{,} \\
V_{j} &=&\hat H_{j}+2\underset{i=j+1}{\overset{q_1}{\sum }}\hat H_{i}\text{,}
\quad W_{j} = \hat H_{-j}+2\underset{i=j+1}{\overset{q_1}{\sum }}\hat H_{-i}%
\text{,} \\
B &=&T(\hat H_{-q_1},\ldots ,\hat H_{q_1})-H_{\mathrm{TL}}(\hat H_{2},\ldots
,\hat H_{q_1})-H_{\mathrm{BR}}(\hat H_{-2},\ldots ,\hat H_{-q_1})\text{.}
\end{eqnarray*}%
where $T$ indicates the block Toeplitz matrix, while $H_{\mathrm{TL}}$ and $%
H_{\mathrm{BR}}$ are respectively the top-left block Hankel matrix and the
bottom-right block Hankel matrix as just previously depicted in the unilevel
setting %
and where the block $\hat H_{j}$ is defined, according to (\ref{eq:ARnosym1D}%
), as
\begin{equation}
\hat H_{j}=\left[
\begin{tabular}{c|ccccc|c}
$v_{j,0}$ & \multicolumn{5}{|c|}{$u_j^T$} & $0$ \\ \hline
$v_{j,1}$ &  &  &  &  &  &  \\
$\vdots $ &  &  &  &  &  &  \\
$v_{j,q_2}$ &  &  & $B_j$ &  &  & $w_{j,q_2}$ \\
&  &  &  &  &  & $\vdots $ \\
&  &  &  &  &  & $w_{j,1}$ \\ \hline
$0$ & \multicolumn{5}{|c|}{$-(Ju_j)^T$} & $w_{j,0}$%
\end{tabular}%
\right] \text{,}  \label{eq:2d_form_block}
\end{equation}%
with $B_j=T(h_{j,-q_2},\ldots, h_{j,q_2})-H_{\mathrm{TL}}(h_{j,2},\ldots
,h_{j,q_2}) -H_{\mathrm{BR}}(h_{j,-2},\ldots ,h_{j,-q_2})$.

Refer to (\ref{eq:AR2D}) and (\ref{eq:AR2Dbis}) for the structure of the
matrix $P$ related to a strongly symmetric PSF in which the minimizer $P^*$,
see (\ref{eq:Pstar_optimal}), will be searched.

\begin{theorem}
Let $A=A(h)$ be the Anti-Reflective matrix generated by the generic PSF $%
h_{PSF}=\left[h_{i_1,i_2}\right]_{i_1=-q_1,\ldots,q_1,
i_2=-q_2,\ldots,q_2}$. \newline The optimal preconditioner in the
$\mathcal{AR}_n^{2D}$ algebra is the
matrix associated with the symmetrized PSF $s_{PSF}=\left[s_{i_1,i_2}\right]%
_{i_1=-q_1,\ldots,q_1, i_2=-q_2,\ldots,q_2}$, with
\begin{equation}  \label{eq:simmetrizzata2d-bis}
s_{\pm i_1,\pm i_2}=\dfrac{%
h_{-i_1,-i_2}+h_{-i_1,i_2}+h_{i_1,-i_2}+h_{i_1,i_2}}{4}\text{.}
\end{equation}%
%
%
\end{theorem}

\begin{proof}
The proof can be done by extending the geometrical approach just considered
in the one-dimensional case: we simply observe that if we consider in the $4$%
-dimensional space a point $R=(R_x,R_y,R_z,R_w)$, its optimal approximation $%
Q^*$ among the points $Q=(Q_x,Q_y,Q_z,Q_w)$ belonging to the line $\mathcal{L%
}$
\begin{equation*}
\left\{
\begin{array}{c}
x=t \\
y=t \\
z=t \\
w=t%
\end{array}%
\right.
\end{equation*}%
%
%
is obtained by minimizing the distance
\begin{eqnarray*}
\mathbf{d}^2(\mathcal{L},R)
&=&4t^{2}-2t(R_x+R_y+R_z+R_w)+R_x^{2}+R_y^{2}+R_z^{2}+R_w^{2}\text{.}
\end{eqnarray*}%
This is a trinomial of the form $\alpha t^{2}+\beta t+\gamma$, with $\alpha>0
$ and we find the minimum by using the formula for computing the abscissa of
the vertex of a parabola%
\begin{equation*}
{t}^*=-\frac{\beta}{2\alpha}=\frac{R_x+R_y+R_z+R_w}{4}\text{.}
\end{equation*}%
Hence the point $Q^*$ is defined as $Q^*_x=Q^*_y=Q^*_z=Q^*_w=t^*$. The same
holds true if we consider any swapped point $R^S$, not unique but depending
on the permutation at hand, since they share the same distance, i.e., $%
d(R,Q^*)=d(R^S,Q^*)$. Again, thanks to the linearity of obtained expression,
this result can be extended also in the case of any linear combination of coordinates.

Thus, by explicitly exploiting the structure of the matrices $A$
and $P$, we define a point by referring to the entry with positive
and negative two-index. For instance, points pertaining to the
Toeplitz part are defined as
\begin{eqnarray*}
\Theta_{i_1,i_2} &=&
(\theta_{i_1,i_2}^x,\theta_{i_1,i_2}^y,\theta_{i_1,i_2}^z,%
\theta_{i_1,i_2}^w)=(h_{-i_1,-i_2},h_{-i_1,i_2},h_{i_1,-i_2},h_{i_1,i_2}), \\
\hat \Theta_{i_1,i_2} &=& (\hat \theta_{i_1,i_2}^x,\hat
\theta_{i_1,i_2}^y,\hat \theta_{i_1,i_2}^z,\hat
\theta_{i_1,i_2}^w)=(s_{-i_1,-i_2},s_{-i_1,i_2},s_{i_1,-i_2},s_{i_1,i_2}),
\end{eqnarray*}%
respectively.

As in the unilevel setting, the original minimization problem is transformed
in the equivalent problem of minimizing the sum of squared distances
analogously as in (\ref{eq:points}). We start again from points pertaining
to the Toeplitz part, that can be minimized separately, and we obtain the
minimizer (\ref{eq:simmetrizzata2d-bis}). It is also an easy check to prove
the same claim with respect to any other couple of points pertaining to
Hankel or low rank corrections, by invoking the quoted linearity argument.
\end{proof}

It is worth stressing that this proof idea is very powerful in its
generality. It can be applied to any type of BCs based on the fact that the
values of unknowns outside the FOV are defined as linear combinations of the
unknowns inside the FOV, so that it may be useful in the future to prove
theoretical results for new proposed BCs.

\section{Computational results}
\label{sez:Computational}

A well-known iterative method for solving the image deblurring problem is Landweber method \cite{Landweber}, whose $(k+1)$-th iteration step is defined by
\begin{equation}
\label{LandweberMeth}
x_{k+1}=x_{k}+\tau A^{H}(g-Ax_{k}),
\end{equation}
where $A$ is the blurring matrix, $g$ is the recorded image and
$\tau$ is the descent parameter (we set it equal to one).  As one
can observe experimentally, the restorations seem to converge in
the initial iterations, before they become worse and finally
diverge; this phenomenon is called semiconvergence. Hence
Landweber method is a regularization method, where the number of
steps $k$ is the regularization parameter. Moreover it has good
stability properties, but it is usually very slow to converge to
the sought solution. Therefore it is a good candidate for testing
the proposed preconditioning technique. Thus we introduce the
preconditioned Landweber method
\begin{equation}
\label{DLandweberMeth}
x_{k+1}=x_{k}+\tau DA^{H}(g-Ax_{k}),
\end{equation}
where $D$ is the preconditioner.  In order to build it, we compute
the eigenvalues $\lambda _{j}$ of the blurring matrix associated
to the PSF and to periodic BCs (via FFT) or to the symmetrized PSF
and to Reflective BCs (via FCT) or to the symmetrized PSF and to
Anti-Reflective BCs (via FST, see Theorem \ref{teo:teo_Jordan} and
Theorem \ref{teo:teo_autovalori} and comments below), then we
apply the Tikhonov Filter
\begin{equation}
\begin{array}{cc}
d_{j}=\dfrac{1}{\left\vert \lambda _{j}\right\vert ^{2}+\alpha }
\end{array}
\label{TikFilter}
\end{equation}
to determine the eigenvalues $d_{j}$ of $D$; finally the PSF
related to $D$ can be obtained via IFFT or IFCT or IFST (the
inverses of the previous transforms, namely inverse FFT, inverse
FCT, inverse FST)). In numerical experiments we have set the
parameter $\alpha$ manually, so that we have reached excellent
performances both in terms of quality of the restorations and in
acceleration of the method.

Actually in our implementation, which is partially based on the
Matlab Toolbox \textsf{RestoreTools} \cite{NagyTools}, we have
never worked with $A^{H}$, but always with $A^{\prime }$, that is
the matrix related to the PSF rotated by 180 degrees. This
approach is known in literature as reblurring strategy
\cite{DS-IP-2005}. The reason behind this choice resides in one of
the main problems of Anti-Reflective algebra $\mathcal{AR}$, i.e.
the fact that it is not closed under transposition. We stress that
$A^{H}$ and $A^{\prime }$ are the same thing in case of periodic
and zero boundary conditions, but they are different for
Reflective and Anti-Reflective ones.

To test these different BCs and preconditioning techniques,  we
have taken into account the Cameraman deblurring problem of Figure
\ref{Fig:Cameraman}, in which the PSF is a slightly non-symmetric
portion of a Gaussian blur, and the Bridge deblurring problem of
Figure \ref{Fig:Bridge}, in which the PSF is an highly
non-symmetric portion of a Gaussian blur. In both cases we have
generated the blurred and noisy data $g$, adding about $0.1\%$ of
white Gaussian noise. We have chosen to add a low level of noise
to emphasize the importance of boundary conditions, which play a
leading role when the noise is low, while they become less
decisive when it grows up. Since we know the true image $f$, to
measure the quality of the deblurred images we compute the
Relative Restoration Error (RRE) $\left\Vert x-f\right\Vert _{\cal
F}/\left\Vert f\right\Vert_{\cal F}$, where $\left\Vert \cdot
\right\Vert _{\cal F}$ is the Frobenius norm and $x$ is the
computed restoration.

As we expected, from Table \ref{Tab:quasi_simm} and Table
\ref{Tab:non_simm}, we can notice that both Reflective and
Anti-Reflective boundary conditions outperform periodic ones,
which give rise to poor restorations (see first image of Figure
\ref{Fig:Cameraman_Landweber} and Figure
\ref{Fig:Bridge_Landweber}). Furthermore by means of
Anti-Reflective BCs (see third image of Figure
\ref{Fig:Cameraman_Landweber} and Figure
\ref{Fig:Bridge_Landweber}) we can gain restorations of better
quality compared with ones obtained employing Reflective BCs (see
second image of Figure \ref{Fig:Cameraman_Landweber} and Figure
\ref{Fig:Bridge_Landweber}). From Tables
\ref{Tab:quasi_simm}-\ref{Tab:non_simm} and Figures
\ref{Fig:Cameraman_D_Landweber}-\ref{Fig:Bridge_D_Landweber} we
can see that all these considerations hold also for $D$-Landweber
method --- i.e. Landweber method with preconditioning --- which
for a suitable choice of the parameter $\alpha$ is able to reach
restorations of the same quality of the classical Landweber method
in much smaller number of steps. In particular the reduction in
steps for both Reflective and Anti-Reflective BCs is around $50$
times for the Cameraman deblurring problem and around $8$ for the
Bridge deblurring problem.

We stress that the iteration count reported in Table
\ref{Tab:non_simm} in the Anti-Reflective row  does not have to
deceive, because, as it can be seen from Figure
\ref{Fig:Bridge_Grafici}, if we compare the restorations gained by
Landweber (preconditioned or not) at any given fixed iteration,
employing Reflective BCs or Anti-Reflective BCs, we see that the
latter shows always equal or better restoration quality. The same
remark holds for the Cameraman deblurring problem (see Table
\ref{Tab:quasi_simm}). In fact Figure \ref{Fig:Bridge_Grafici} is
very instructive because it tells to the generic user two things:
a) the curves for Reflective and Anti-Reflective BCs are very
flat, b) the approximation obtained when using Anti-Reflective BCs
is always better or equal to that obtained with Reflective BCs.
The combined message of the previous two items is that we can
safely choose the Anti-Reflective BCs, even when we are unable to
estimate precisely the stopping criterion for deciding the optimal
iteration: we notice that this observation does not hold for the
periodic BCs where a small error in the evaluation of the optimal
iteration leads to a substantial deterioration of the quality of
the resulting restored image.

In the end, from the results reported in this section we can say
that our proposal of the optimal preconditioner in the context of
Anti-Reflective BCs is as effective as the one introduced in
\cite{NCT-SISC-1999} for Reflective BCs. Therefore the present
work represents a theoretical and numerical continuation and
strengthening of that line of research.

\begin{figure}[thp]
\centering
\begin{tabular}{ccc}
\includegraphics[scale=0.5]{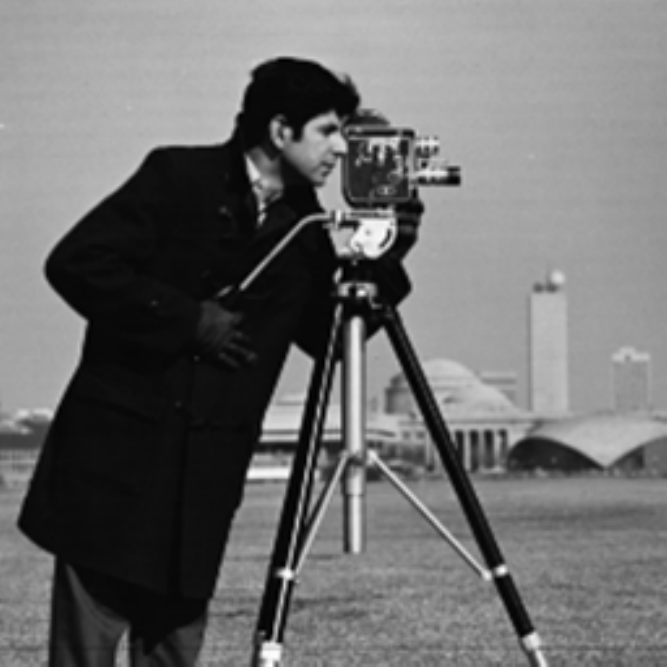} &
\includegraphics[scale=0.5]{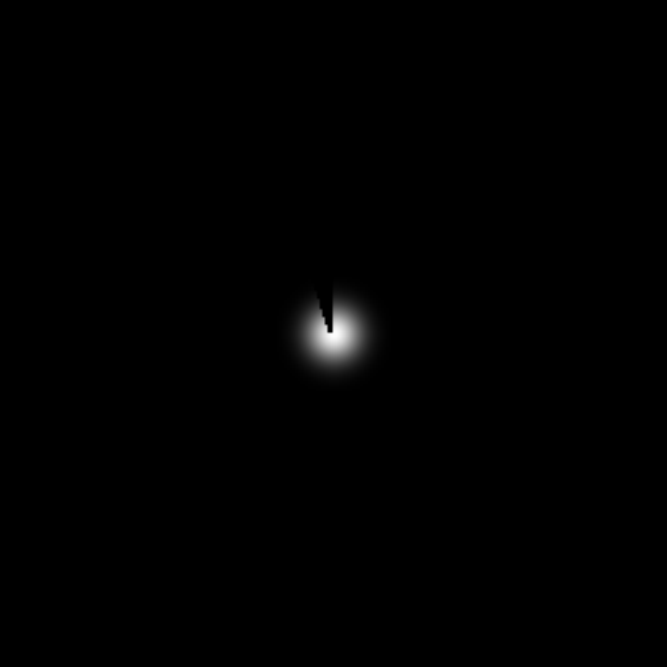} &
\includegraphics[scale=0.5]{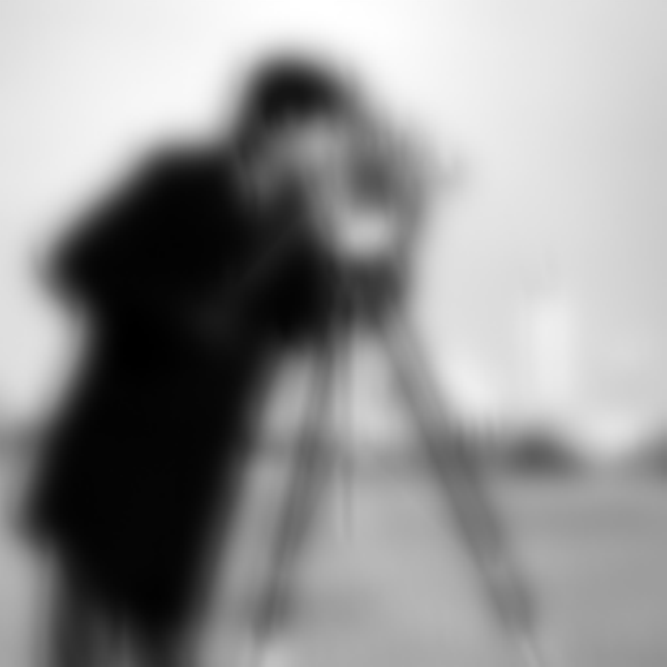}
\end{tabular}%
\caption{Cameraman deblurring problem: true image, PSF, blurred and noisy image.}
\label{Fig:Cameraman}
\end{figure}

\begin{table}[thp]
\centering
\begin{tabular}{c|c|c|c|c|}
\cline{2-5}
& \multicolumn{2}{|c}{Landweber} & \multicolumn{2}{|c|}{$D$-Landweber} \\ \cline{2-5}
& RRE & IT & RRE & IT \\ \hline
\multicolumn{1}{|r|}{Periodic} & 0.2147 & 46 & 0.2136 & 3
\\
\multicolumn{1}{|r|}{Reflective} & 0.1611 & 953 & 0.1611 & 19
\\
\multicolumn{1}{|r|}{Anti-Reflective} & 0.1582 & 1461 & 0.1582 & 25
\\
\hline
\end{tabular}
\caption{Results of the classical and preconditioned Landweber method related to the Cameraman deblurring problem, employing different BCs.}
\label{Tab:quasi_simm}
\end{table}

\begin{figure}[thp]
\centering
\begin{tabular}{ccc}
\includegraphics[scale=0.5]{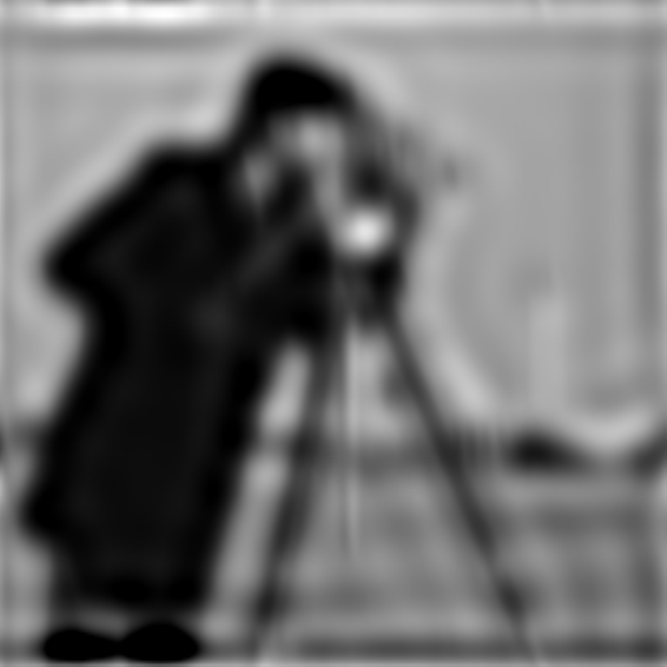} &
\includegraphics[scale=0.5]{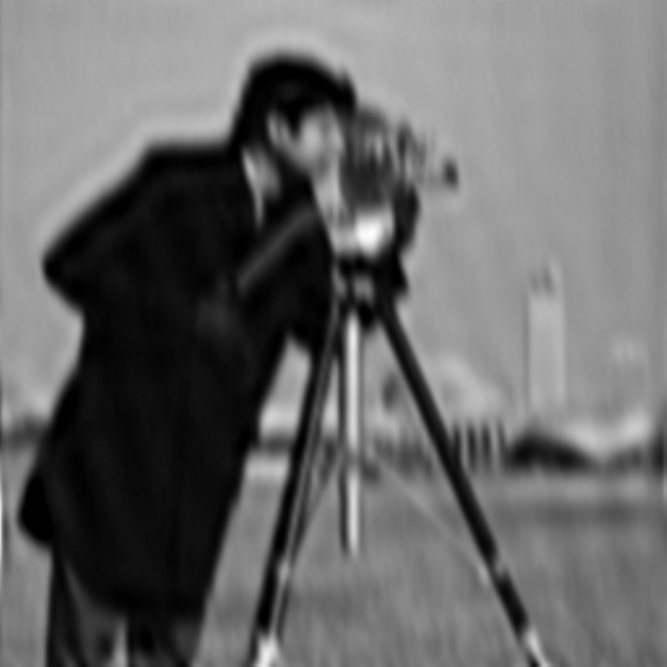} &
\includegraphics[scale=0.5]{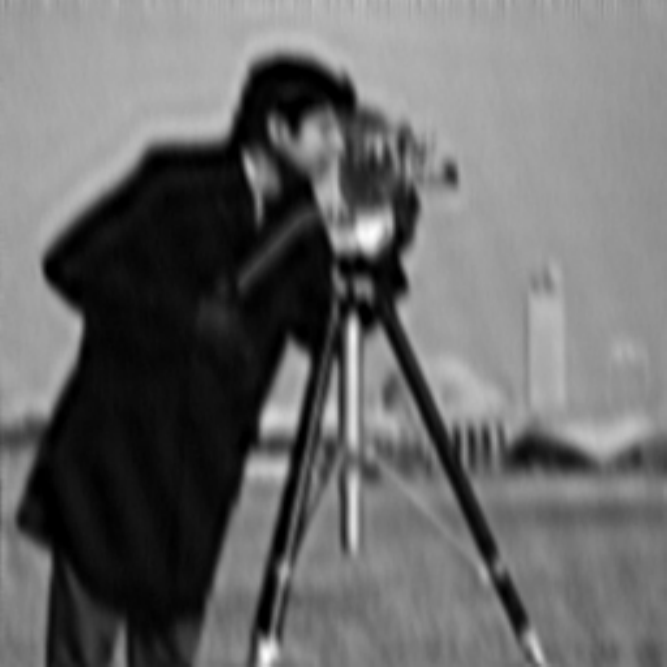}
\end{tabular}%
\caption{Landweber restorations, employing periodic, reflective, anti-reflective BCs.}
\label{Fig:Cameraman_Landweber}
\end{figure}

\begin{figure}[thp]
\centering
\begin{tabular}{ccc}
\includegraphics[scale=0.5]{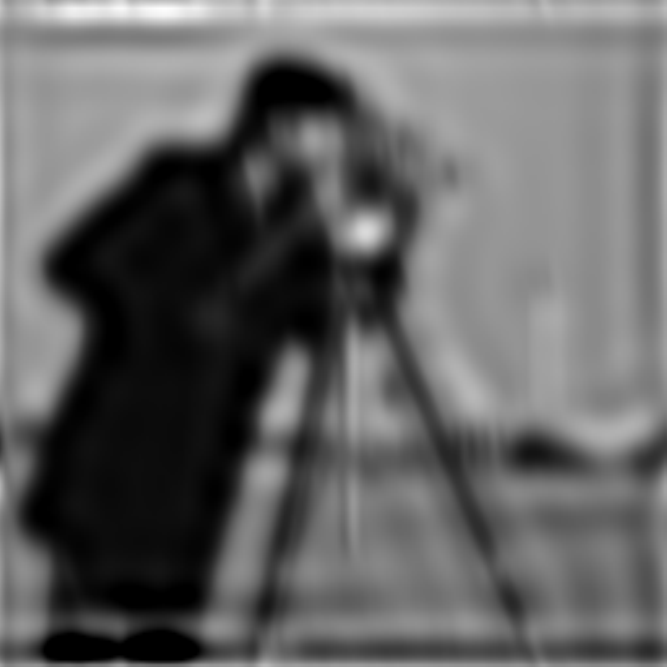} &
\includegraphics[scale=0.5]{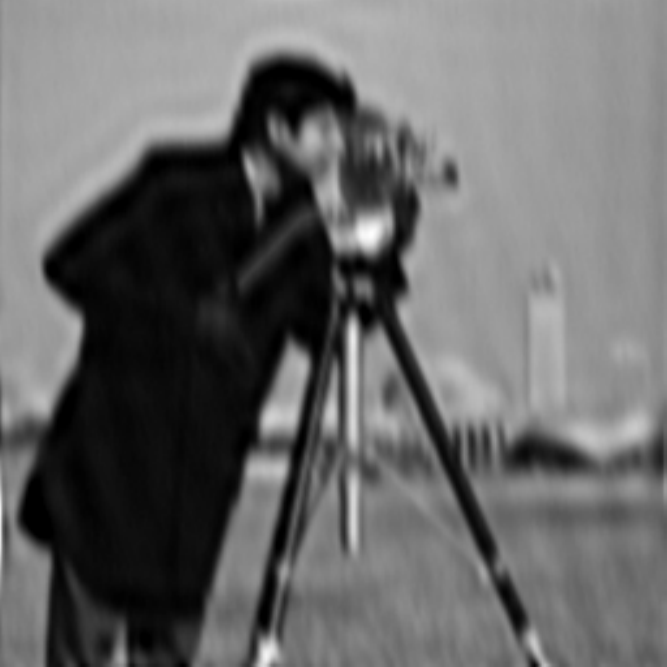} &
\includegraphics[scale=0.5]{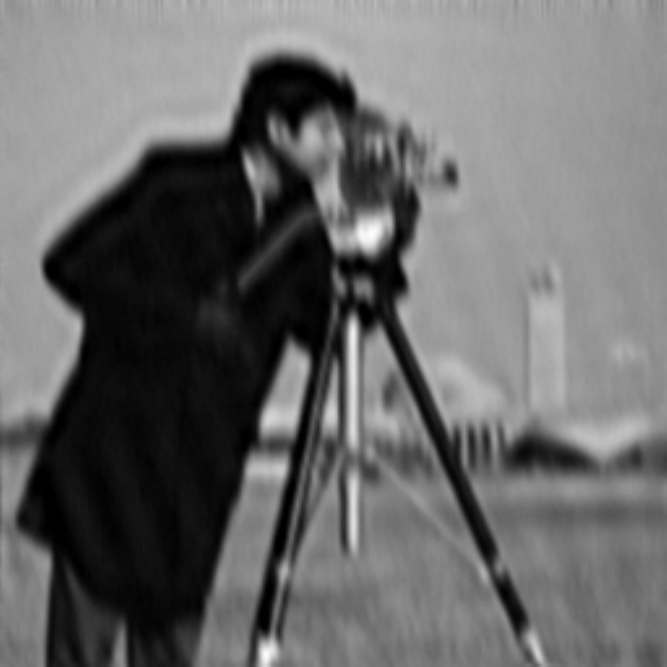}
\end{tabular}%
\caption{Preconditioned Landweber restorations, employing periodic, reflective, anti-reflective BCs.}
\label{Fig:Cameraman_D_Landweber}
\end{figure}

\begin{figure}[thp]
\centering
\begin{tabular}{ccc}
\includegraphics[scale=0.5]{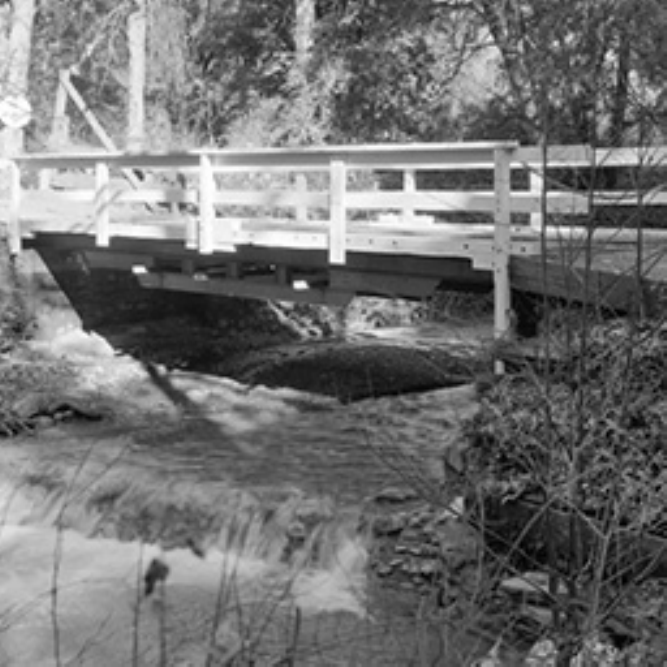} &
\includegraphics[scale=0.5]{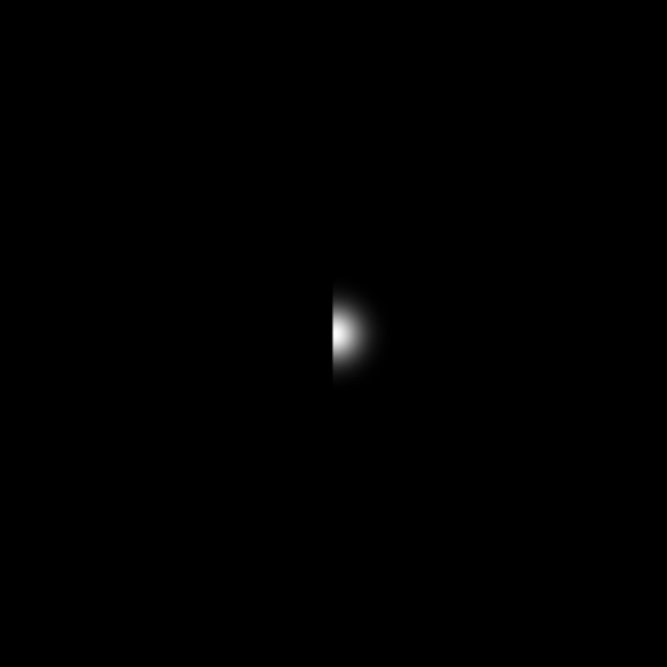} &
\includegraphics[scale=0.5]{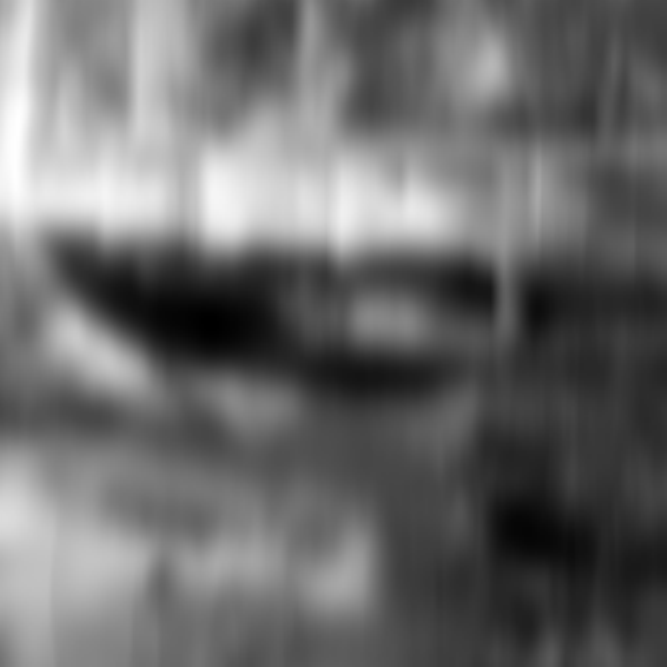}
\end{tabular}%
\caption{Bridge deblurring problem: true image, PSF, blurred and noisy image.}
\label{Fig:Bridge}
\end{figure}

\begin{table}[thp]
\centering
\begin{tabular}{c|c|c|c|c|}
\cline{2-5}
& \multicolumn{2}{|c}{Landweber} & \multicolumn{2}{|c|}{$D$-Landweber} \\ \cline{2-5}
& RRE & IT & RRE & IT \\ \hline
\multicolumn{1}{|r|}{Periodic} & 0.2573 & 9 & 0.2561 & 2
\\
\multicolumn{1}{|r|}{Reflective} & 0.2195 & 1281 & 0.2195 & 146
\\
\multicolumn{1}{|r|}{Anti-Reflective} & 0.2114 & 12824 & 0.2114 & 1718
\\
\hline
\end{tabular}
\caption{Results of the classical and preconditioned Landweber method related to the Bridge deblurring problem, employing different BCs.}
\label{Tab:non_simm}
\end{table}

\begin{figure}[thp]
\centering
\begin{tabular}{ccc}
\includegraphics[scale=0.5]{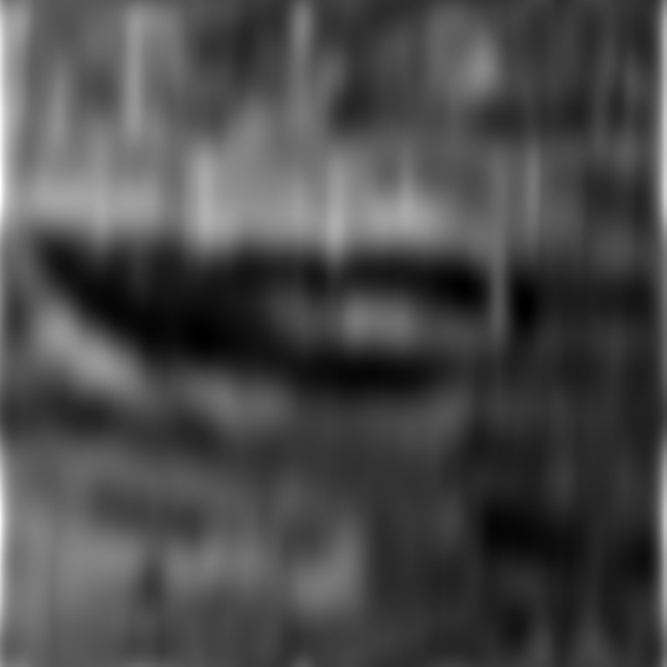} &
\includegraphics[scale=0.5]{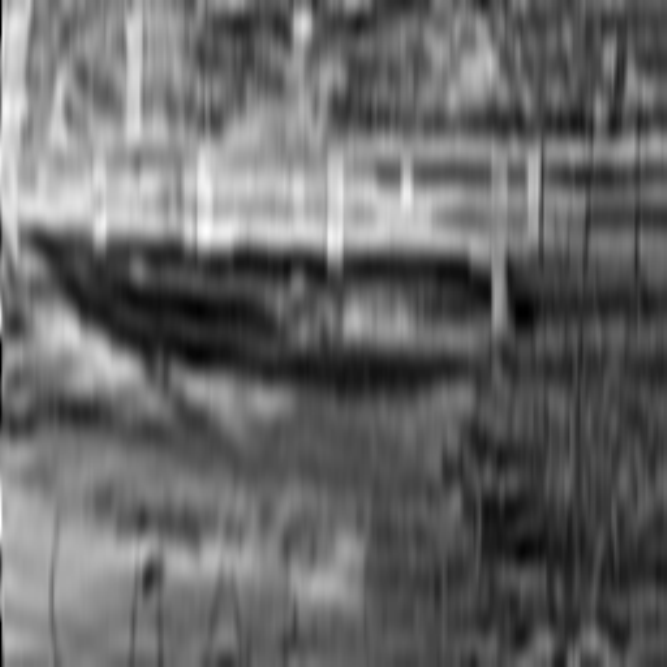} &
\includegraphics[scale=0.5]{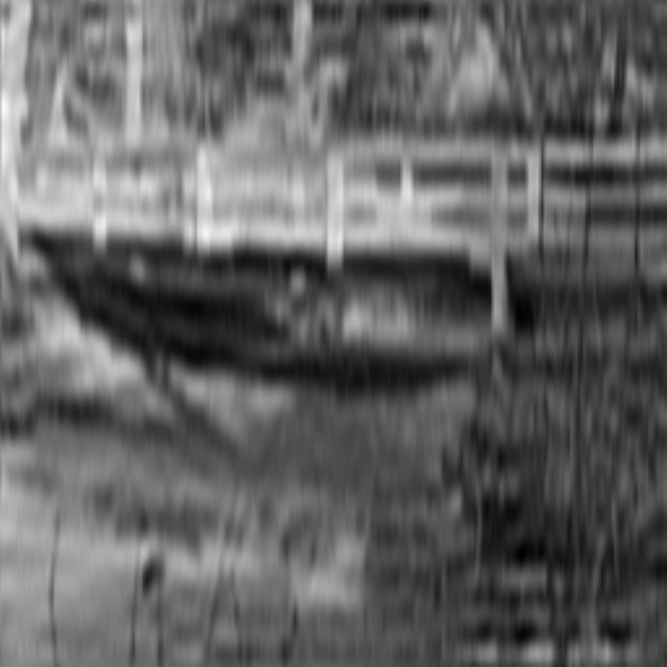}
\end{tabular}%
\caption{Landweber restorations, employing periodic, reflective, anti-reflective BCs.}
\label{Fig:Bridge_Landweber}
\end{figure}

\begin{figure}[thp]
\centering
\begin{tabular}{ccccc}
\includegraphics[scale=0.5]{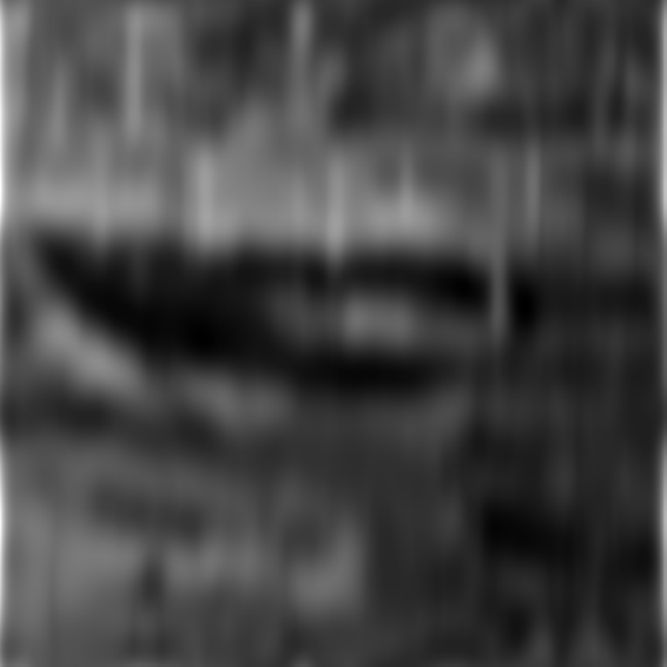} &
\includegraphics[scale=0.5]{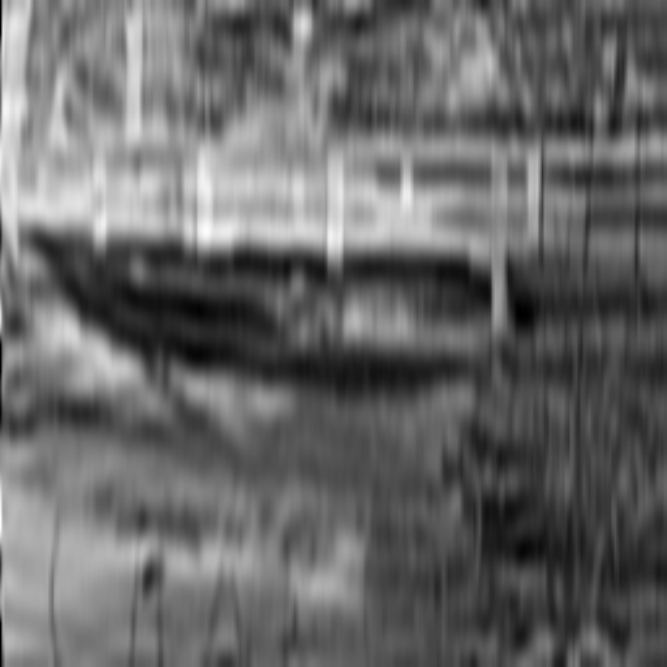} &
\includegraphics[scale=0.5]{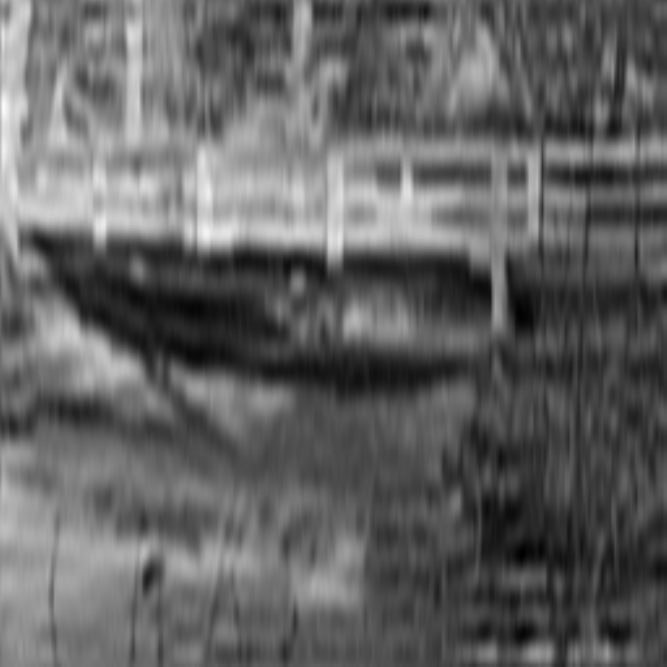}
\end{tabular}%
\caption{Preconditioned Landweber restorations, employing periodic, reflective, anti-reflective BCs.}
\label{Fig:Bridge_D_Landweber}
\end{figure}

\begin{figure}[htb]
\centering
\begin{tabular}{cc}
\includegraphics[scale=0.4]{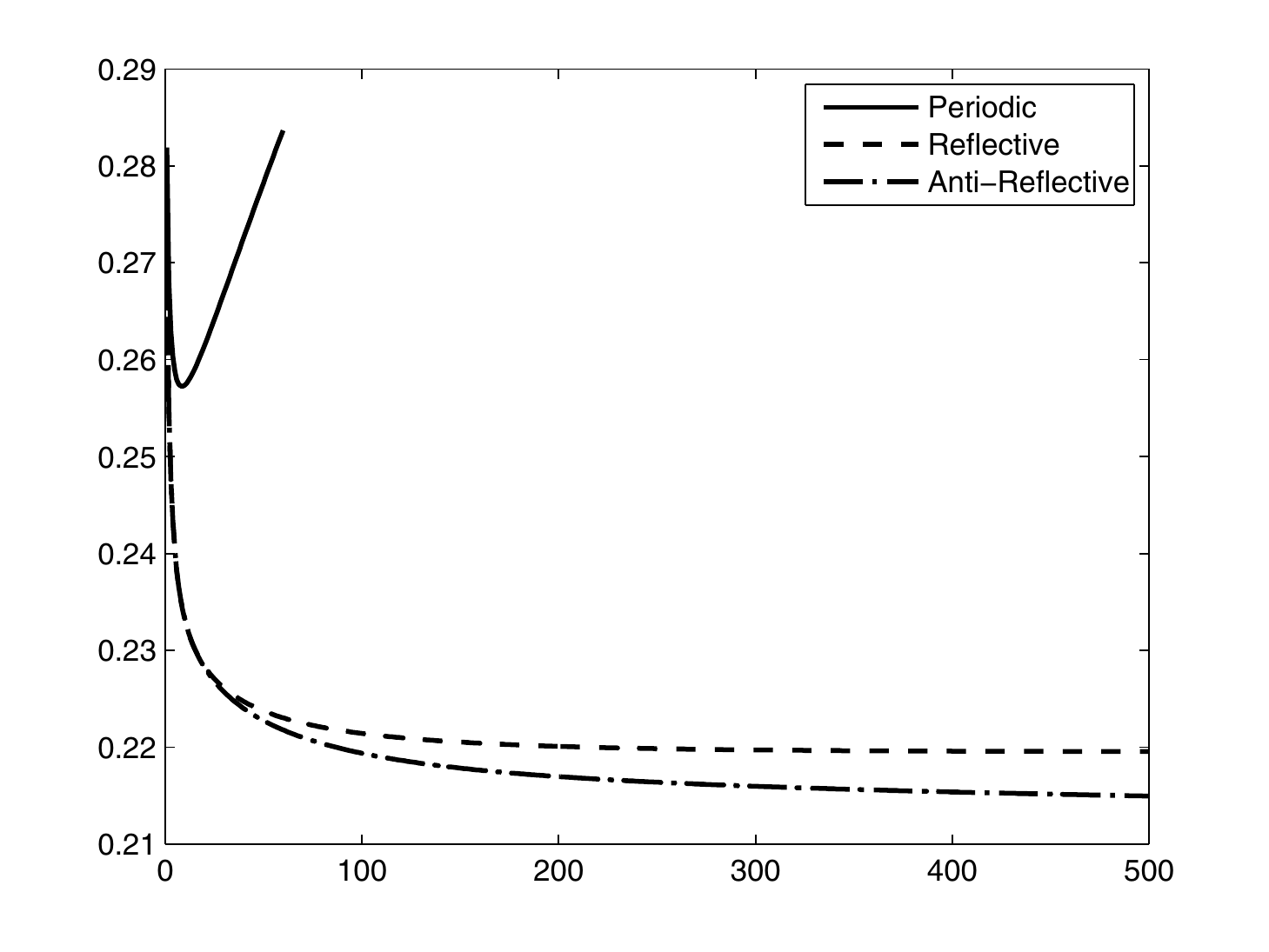}
\includegraphics[scale=0.4]{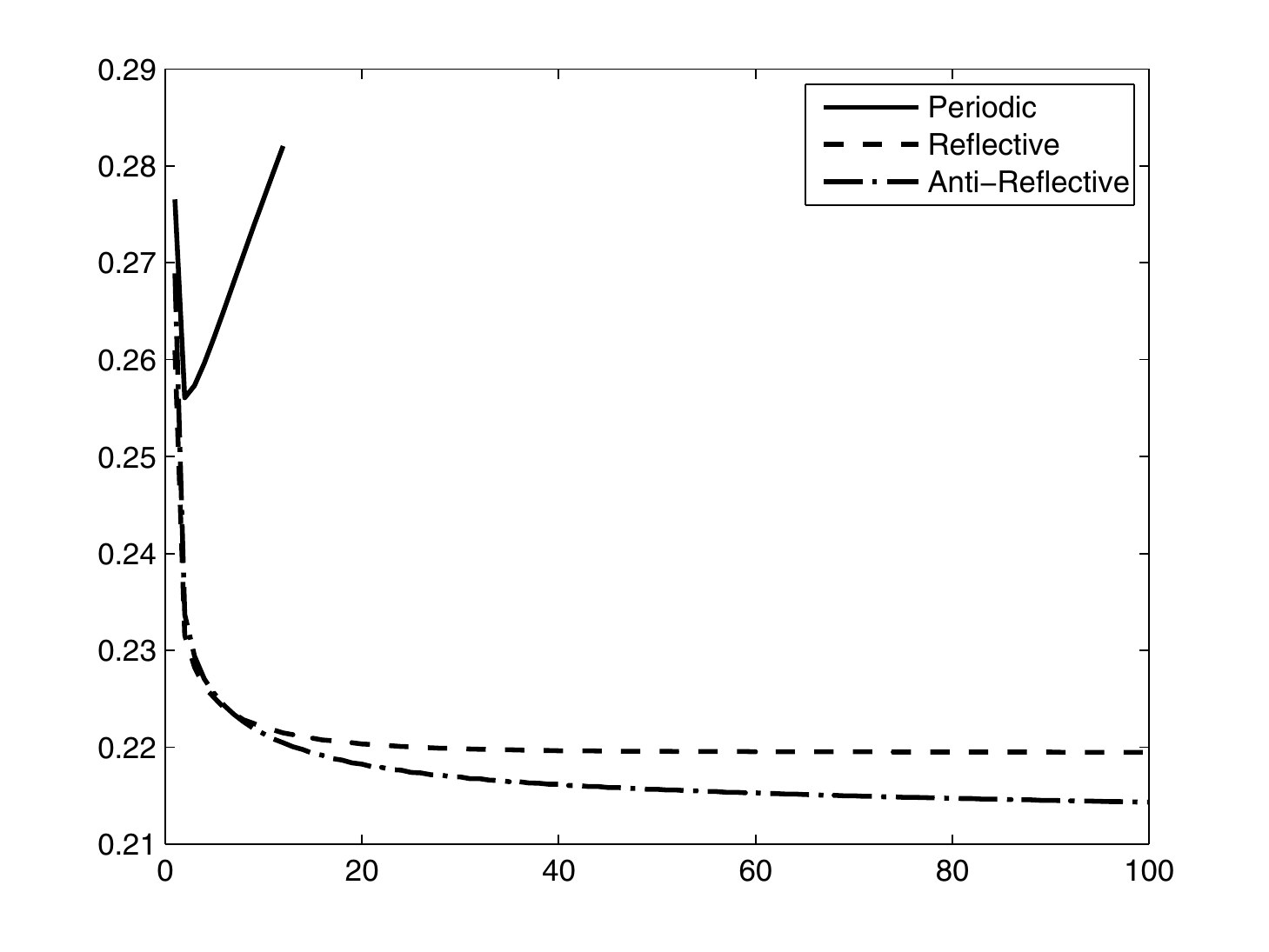}
\end{tabular}%
\caption{Bridge deblurring problem: RRE trends of Landweber (on the left) and $D$-Landweber (on the right) for different BCs.}
\label{Fig:Bridge_Grafici}
\end{figure}

\section{Conclusions and Perspectives}
\label{sez:Conclusions}

Inspired by the theoretical results on optimal preconditioning
stated in \cite{NCT-SISC-1999} in the Reflective BCs environment,
in this paper we have presented analogous results for
Anti-Reflective BCs. In both cases the optimal preconditioner is
the blurring matrix associated to the symmetrized PSF. We stress
that our proof is based on a geometrical idea, which allows to
greatly simplify the used arguments, even when non orthogonal
transforms are involved. Moreover that idea is very powerful in
its generality and it may be useful in the future to prove
theoretical results for new BCs.

Computational results have shown that the proposed preconditioning
strategy is effective and it is able  to give rise to a meaningful
acceleration both for slightly and highly non-symmetric PSFs. On
the other hand, symmetrization is efficient when we have a PSF
that is near to be symmetric and it becomes more and more
ineffective as the PSF departs from symmetry. In this case, other
techniques \cite{Z2}, which  can manage directly non-symmetric
structures, can gain better performances and in this direction we
see a substantial development in the near future.


\end{document}